\documentclass[11pt,reqno]{amsart}
\usepackage{amsmath}
\usepackage{cases}
\usepackage{mathrsfs}
\usepackage{bbm}
\usepackage{amssymb}
\usepackage{amscd}
\usepackage{amsfonts,latexsym,amsmath,
amsthm,amsxtra,mathdots,amssymb,latexsym,mathabx}
\usepackage[all,cmtip]{xy}
\RequirePackage{amsmath} \RequirePackage{amssymb}
\usepackage{color}
\usepackage{colordvi}
\usepackage{multicol}
\usepackage{hyperref}
\usepackage{mathtools}
\usepackage[margin=1in]{geometry}
\usepackage{xcolor}

\hypersetup{
    colorlinks,
    linkcolor={red!50!black},
    citecolor={blue!100!black},
    urlcolor={blue!100!black}
}
\usepackage{cite}

\newcommand{\bea}{\begin{eqnarray}}
\newcommand{\eea}{\end{eqnarray}}
\newcommand{\bna}{\begin{eqnarray*}}
\newcommand{\ena}{\end{eqnarray*}}

\numberwithin{equation}{section}

\theoremstyle{plain}
\newtheorem{lemma}{Lemma}[section]
\newtheorem{theorem}[lemma]{Theorem}

\newtheorem{proposition}[lemma]{Proposition}

\theoremstyle{definition}

\newtheorem{remark}{Remark}

\renewcommand{\Re}{\operatorname{Re}}
\renewcommand{\Im}{\operatorname{Im}}

\renewcommand{\bmod}[1]{\,(\mathrm{mod\,}{#1})}

\makeatletter
\@namedef{subjclassname@2020}{%
  \textup{2020} Mathematics Subject Classification}

\title[]
{On a level analog of Selberg's result on $S(t)$}

\author{Qingfeng Sun}
	\address{School of Mathematics and Statistics, Shandong University, Weihai\\Weihai, 264209, China}
	\email{qfsun@sdu.edu.cn}

\author{Hui Wang}
 \address{Department of Mathematics, The University of Hong Kong, Pokfulam Road, Hong Kong}
    \email{wh0315@mail.sdu.edu.cn}
     


\thanks{Q. Sun was partially supported by the National Natural Science Foundation of China
(Grant Nos. 12471005 and 12031008) and
the Natural Science Foundation of Shandong Province (Grant No. ZR2023MA003).}

\date{}

\begin{document}
\begin{abstract}
Let $S(t,f)=\pi^{-1}\arg L(1/2+it, f)$, where $f$ is a holomorphic Hecke cusp form of weight $2$ and prime level $q$.
{In this paper, we establish an unconditional asymptotic formula for the moments of $S(t,f)$,
providing a level aspect analogue of Selberg's classical work on $S(t)$.
As a consequence, we derive a weighted central limit theorem for the distribution of $S(t,f)$ normalized by $\sqrt{\log\log q}$.
To this end, we develop a precise approximation for $S(t,f)$ via a truncated Dirichlet series and employ a weighted zero-density estimate for the corresponding family of $L$-functions.}
\end{abstract}
\maketitle

\section{Introduction}
\setcounter{equation}{0}
For $T>0$, let $N(T)$ be the number of zeros $\rho=\beta+i\gamma$ of the Riemann zeta function $\zeta(s)$ in the rectangle $0<\beta<1$, $0<\gamma\leq T$.
In particular, Davenport \cite[Chapter 15]{Davenport} gives that the approximate formula of $N(T)$ is
\bna
N(T)=\frac{T}{2\pi}\log \frac{T}{2\pi}-\frac{T}{2\pi}+\frac{7}{8}+S(T)+O(T^{-1}),
\ena
where $$S(T)=\frac{1}{\pi}\arg \zeta\big(\frac{1}{2}+iT\big).$$
The argument of $\zeta\big(1/2+iT\big)$ is defined as the continuous variation along the horizontal
segment from $\infty+iT$ to $1/2+iT$ starting at $\infty+iT$ with the value $0$.
Understanding the complicated behavior of $S(t)$ is an interesting topic in analytic number theory,
and a series of results have emerged in recent years due to Selberg's pioneering work.

In \cite{Selberg, Selberg1}, Selberg showed that for any $n\in \mathbb{N}$ one has
\bea\label{S(t) even moment}
\int_{T}^{2T}S(t)^{2n}\mathrm{d}t=\frac{(2n)!}{n!(2\pi)^{2n}}T(\log\log T)^n+O\big(T(\log\log T)^{n-1/2}\big),
\eea
which can imply, on average sense, $|S(t)|$ has order of magnitude $\sqrt{\log\log T}$.
The same method can be applied to get a similar result relating to Dirichlet $L$-functions $L(s,\chi)$,
where $\chi$ is the primitive Dirichlet character of modulus $q$.
More precisely, for prime modulus $q$ and $t>0$, Selberg \cite{Selberg2} showed that
\bea\label{S(chi) even moment}
\sideset{}{^*}\sum_{\chi \bmod q} S(t, \chi)^{2n}=\frac{(2n)!}{n!(2\pi)^{2n}}q(\log\log q)^n+O\big(q(\log\log q)^{n-1/2}\big),
\eea
where $*$ denotes the summation goes through the primitive characters $\chi \bmod q$,
and $S(t,\chi)$ is defined by $$S(t,\chi)=\frac{1}{\pi}\arg L\big(\frac{1}{2}+i t, \chi\big).$$
Note that the counterparts of \eqref{S(t) even moment} and \eqref{S(chi) even moment} for the odd-order moments
can also be obtained by making a minor adaptation in Selberg's techniques.

For higher rank $L$-functions, there are some analogues of \eqref{S(chi) even moment} in the spectral and weight aspects, respectively.
To provide context for these results and our own work, we first recall some fundamental concepts.
The notation $\mathrm{GL}_{n}$ denotes the general linear group of $n \times n$ invertible matrices,
while $\mathrm{SL}_{n}(\mathbb{R})$(resp. $\mathrm{SL}_{n}(\mathbb{Z})$) is
the special linear group of $n \times n$ real matrices (resp. integer matrices) with determinant 1. For a positive integer $q$, the Hecke congruence subgroup of level $q$ is defined as
\[
\Gamma_0(q) = \left\{ \begin{pmatrix} a_{11} & \cdots & a_{1n} \\ \vdots & \ddots & \vdots \\ a_{n1} & \cdots & a_{nn} \end{pmatrix} \in \mathrm{SL}_{n}(\mathbb{Z}) : a_{ij} \equiv 0 \bmod q \text{ for all } i>j\right\}.
\]
For $n = 2$, this group acts on the upper half-plane $\mathbb{H} = \{z=x+iy \in \mathbb{C} : y > 0\}$ via M\"obius transformations $z \mapsto (az+b)/(cz+d)$; for general $n$, the corresponding symmetric space is $\mathrm{SL}_{n}(\mathbb{R})/\mathrm{SO}_n$, with $\textrm{SO}_n$ the group of all $n\times n$ real orthogonal matrices with determinant 1.
Within this framework, and under Generalized Riemann Hypothesis (GRH), Hejhal and Luo \cite{HL} (resp. Liu and Liu \cite{LL}) obtained an asymptotic
formula for the spectral moments of $S_F(t)=\pi^{-1}\arg L\big(\frac{1}{2}+it, F\big)$ for each fixed $t>0$, where
$F$ is the Hecke--Maass cusp form for $\rm SL_2(\mathbb{Z})$ (resp. $\rm SL_3(\mathbb{Z})$)
with eigenvalue $\lambda_F=1/4+t_F^2$ (resp. the spectral parameter $\nu_F=(\nu_{F,1}, \nu_{F,2}, \nu_{F,3})$).
Recently, when $f$ is a holomorphic Hecke cusp form of $\rm SL_2(\mathbb{Z})$ with weight $k$,
Liu and Shim \cite{LS} unconditionally establish the asymptotic formula for the moments of
$S(t,f)=\frac{1}{\pi}\arg L\big(\frac{1}{2}+it, f\big)$,
by means of the weighted version of the zero-density estimate obtained by Hough \cite{Hough} in the weight aspect.
Subsequently, the authors \cite{SW} (resp. \cite{SW1}) employ similar
arguments to eliminate the dependence on the GRH in \cite{HL} (resp. \cite{LL})
and obtain unconditional asymptotic formulas for higher-order moments.

In this paper, we are devoted to get an unconditional analog for $\rm GL_2$ in the level aspect without assuming the GRH.
Assume $q$ is a fixed large prime.
Let $S_2(q)$ be the space of holomorphic cusp forms of weight $2$ and level $q$
equipped with the Petersson inner product
\bna
\left<f,g\right>=\int_{\Gamma_0(q)\setminus\mathbb{H}}f(z)\overline{g(z)}\frac{\mathrm{d}x\mathrm{d}y}{y^2}.
\ena
Here $\Gamma_0(q)\setminus\mathbb{H}$ is a fundamental domain of the action of the Hecke congruence subgroup
$\Gamma_0(q)$ on the upper-half plane $\mathbb{H}$.
Let $H_2(q)$ be a basis of
$S_2(q)$ that are eigenfunctions of all the Hecke operators and have the first Fourier coefficient $\lambda_f(1)=1$.
For each $f\in H_2(q)$, let
\bna
f(z)=\sum_{n=1}^\infty n^{1/2}\lambda_f(n)e(nz), \quad\,\, z\in\mathbb{H}
\ena
be its Fourier expansion.
The Fourier coefficients $\lambda_f(n)$ are real and satisfy
the Deligne's bound (in particular, this is due to Eichler--Shimura \cite{Shimura} in the case of weight $2$)
\bea\label{Deligne bound}
|\lambda_f(n)|\leq \tau(n),
\eea
with $\tau(n)$ being the divisor function.
For $\Re(s)>1$, the Hecke $L$-function associated to $f\in H_2(q)$ is defined by
\bna
L(s,f)=\sum_{n=1}^\infty\frac{\lambda_f(n)}{n^s}=\prod_p(1-\lambda_f(p)p^{-s}+\varepsilon_q(p)p^{-2s})^{-1},
\ena
where $\varepsilon_q(\cdot)$ is the principal Dirichlet character modulo $q$.
It can be analytically continued to the whole complex plane and has the completed $L$-function
\bna
\Lambda(s,f)=\left(\frac{\sqrt{q}}{2\pi}\right)^{s}
\Gamma\left(s+\frac{1}{2}\right)L(s,f),
\ena
which satisfies the functional equation
\bna\label{FE}
\Lambda(s,f)=\epsilon_f\Lambda(1-s,f),
\ena
where $\epsilon_f=q^{1/2}\lambda_f(q)\in\{\pm 1\}$ is the sign of the functional equation.
Moreover, we have the following multiplicative property of the coefficients $\lambda_f(n)$,
\bea\label{the Hecke relation}
\lambda_f(m)\lambda_f(n)
=\sum\limits_{d|(m,n)}\varepsilon_q(d)\lambda_f\left(\frac{mn}{d^2}\right),
\quad \text{for}\,\, m,n\geq 1.
\eea
If $p$ is a prime $\neq q$, we write $1-\lambda_f(p)Y+Y^2=(1-\alpha_f(p)Y)(1-\beta_f(p)Y)$,
so
\bna
\lambda_f(p)=\alpha_f(p)+\beta_f(p).
\ena
The bound \eqref{Deligne bound} is equivalent (for $n$ coprime with the level) with the assertion that $|\alpha_f(p)|=1$
for all $p\neq q$. For $p=q$, the $p$-factor of $L(s,f)$ is of degree at most 1, and we let $\alpha_f(p)=\lambda_f(p)$,
which is shown to be of modulus at most $1$, and $\beta_f(p)=0$.

Now we define $S(t,f)$ by
\bna
S(t,f):=\frac{1}{\pi}\arg L(1/2+it, f),
\ena
where the argument $\arg L(1/2+it, f)$ is obtained by continuous variation along the line $\Im(s)=t$ from
$\sigma=+\infty$ to $\sigma=1/2$.
Our main result is the following theorem.
\begin{theorem}\label{main-theorem}
Let $t>0$ and $n\in \mathbb{N}$ be given.
For sufficiently large prime number $q$, we have
\begin{equation*}
\begin{split}
\sum_{f\in H_2(q)}\omega_f\cdot S(t,f)^n=
C_n(\log\log q)^{n/2}+O_{t,n}\big((\log\log q)^{(n-1)/2}\big),
\end{split}
\end{equation*}
where the harmonic weight $\omega_f$ is defined by \eqref{the harmonic weight} below and $C_n$ is defined by
\bna
C_n=\begin{cases}
\frac{n!}{(n/2)!(2\pi)^{n}}, \,&\textit{if }\, n\,\, \text{is even},\\
0, \,& \textit{if }\, n \,\,\text{is odd}.
\end{cases}
\ena
\end{theorem}
\begin{remark}
For a holomorphic Hecke cusp form $f$ with weight $2$ and level prime number $q$,
Theorem \ref{main-theorem} yields the asymptotic
\[
\sum_{f \in H_2(q)}\omega_f\cdot S(t,f)^{2m} \sim \frac{(2m)!}{m!(2\pi)^{2m}} (\log\log q)^m, \quad m\in \mathbb{N},
\]
as $q\rightarrow\infty$. Together with Lemma \ref{corollary} below, which shows that
the harmonic weight $\omega_f$ satisfies
\[\sum_{f \in H_2(q)} \omega_f =1+o(1),\]
we conclude that $|S(t,f)|$ on average has order of magnitude $\sqrt{\log\log q}$.
\end{remark}
It is another interesting problem to consider the following probability measure
$\mu_q$ on $\mathbb{R}$, defined by
\bna
\mu_q(E):=\left(\sum_{f\in H_2(q)}\omega_f\cdot
\textbf{1}_E\left(\frac{S(t,f)}{\sqrt{\log\log q}}\right)\right)\bigg/
\left(\sum_{f\in H_2(q)}\omega_f\right),
\ena
where $\textbf{1}_E$ is the characteristic function on a Borel measurable set $E$ in $\mathbb{R}$.
As a consequence of Theorem \ref{main-theorem}, we obtain the following weighted central limit theorem.
\begin{theorem}\label{main-corollary}
As $q\rightarrow \infty$, the probability measure $\mu_q$ converges to the Gaussian distribution of mean $0$
and variance $(2\pi^2)^{-1}$; that is
\bna
\lim_{q\rightarrow \infty} \mu_q([a,b])=\sqrt{\pi}\int_a^b \exp(-\pi^2\xi^2)\mathrm{d}\xi
\ena
for any $a<b$.
\end{theorem}
\noindent{\bf Notation.} Throughout the paper,
$\varepsilon$ and A are arbitrarily small and arbitrarily large positive
real numbers, which may be different at each occurrence.
As usual, $e(z) = \exp (2 \pi i z) = e^{2 \pi i z}$.
We write $f = O(g)$ or $f \ll g$ to mean $|f| \leq cg$ for some unspecified positive constant $c$.
The notation $f\sim g$ means that $f-g=o(g)$.
The symbol $f \asymp g$ means that $f=O(g)$ and $g=O(f)$.

\section{Preliminaries}
\label{prelim}
We firstly collect some classical results as the beginning of this section.
\begin{lemma} Let $x\geq 2$. The following statements hold.\\
(1) There exists a constant $c_1$, such that
\bea\label{classical bound 1}
\sum_{p\leq x}\frac{1}{p}=\log\log x+c_1+O\left(\frac{1}{\log x}\right).
\eea
(2) We have
\bna\label{classical bound 3}
\sum_{p\leq x}\frac{\log p}{p}=\log x+O(1),
\ena
and
\bna\label{classical bound 2}
\sum_{p\leq x}\frac{(\log p)^2}{p}\ll (\log x)^2.
\ena
\end{lemma}
\begin{proof}
See Davenport \cite[pp.\,56--57]{Davenport}.
\end{proof}
The Petersson trace formula is given by the following basic orthogonality relation on $H_2(q)$
(see e.g., Kowalski--Michel \cite{KM}).
\begin{lemma}\label{Petersson trace formula}
Let $m,n\geq 1$. Then
\bna
\sum_{f\in H_2(q)}\omega_f\cdot\lambda_f(m)\lambda_f(n)=\delta_{m,n}-
2\pi \sum_{c\equiv 0\bmod q}\frac{S(m,n;c)}{c}J_1\left(\frac{4\pi\sqrt{mn}}{c}\right),
\ena
where $J_1$ is the Bessel function and $S(m,n;c)$ is a classical Kloosterman sum,
\bea\label{the harmonic weight}
\omega_f:=\frac{1}{4\pi \left<f,f\right>}=\frac{\zeta(2)}{|H_2(q)|}\cdot
\frac{1}{L(1,\mathrm{sym}^2f)}\cdot\left(1+O\left(\frac{(\log q)^4}{q}\right)\right)
\eea
is the harmonic weight of $f$.
\end{lemma}

We record here some properties of Bessel functions will be used. The Bessel function
of the first kind, $J_\nu(x)$ with $\nu\in\mathbb{C}$, has Taylor series about zero given by (see \cite[pg. 918, Eq. 8.440]{GR})
\bna
J_\nu(x)=\left(\frac{x}{2}\right)^\nu\sum_{m=0}^\infty\frac{(-1)^m}{m!\Gamma(\nu+1+m)}\left(\frac{x}{2}\right)^{2m}.
\ena
Specializing to $\nu=k-1$ with $k\geq 1$,
we have the following bound for $x>0$ (see \cite[Sections 2.11, 7.1]{Watson}),
\bna\label{J_k bound}
J_{k-1}(x)\ll_k
\begin{cases}
x^{k-1}, \,&\textit{ if } x\leq1,\\
x^{-1/2}, \,&\textit{ if } x>1.
\end{cases}
\ena
By a general bound for Kloosterman sums (see e.g., \cite[(2.13)]{ILS})
\bna
|S(m,n;c)|\leq (m,n,c)\min\left\{\frac{c}{(n,c)},\frac{c}{(m,c)}\right\}^{\frac{1}{2}}\tau(c),
\ena
and the bound $J_1(x)\ll x$ for $J$-Bessel function, we derive (see e.g., \cite[Corollary 2.3]{ILS})
\begin{lemma}\label{corollary}
For any integers $m,n\geq 1$ with $6\pi\sqrt{mn}\leq q$, we have
\bna
\sum_{f\in H_2(q)}\omega_f\cdot\lambda_f(m)\lambda_f(n)=\delta_{m,n}
+O\left(\frac{\tau(q)}{q^{3/2}}\frac{(m,n,q)(mn)^{1/2}}{\big((m,q)+(n,q)\big)^{1/2}}(\tau (m,n))\right).
\ena
\end{lemma}

We will need the following weighted version of zero-density estimate
which was established by Kowalski--Michel \cite[Theorem 4]{KM1}.
\begin{lemma}\label{zero-density}
Let $q$ be a prime number and $$N(f; \sigma, T):=\#\{\rho=\beta+i\gamma \,|\, L(\rho,f)=0, \sigma<\beta, |\gamma|<T\}.$$
There exists an absolute constant $A>1$ such that for any real number $T$ with $T\geq \frac{1}{\log q}$,
for any $\frac{1}{2}+\frac{1}{\log q}\leq\sigma\leq 1$, and for any $c$ with $0<c<\frac{1}{2}$, one has
\bna
\sum_{f\in H_2(q)}\omega_f\cdot N(f; \sigma, T)\ll (1+T)^Aq^{-2c(\sigma-1/2)}(\log q),
\ena
where the implied constant depends only on $c$.
\end{lemma}
Finally, we need the lemma of the theory of moments in probability theory (see, for example, \cite[Theorem 30.2]{Billingsley}).
\begin{lemma}\label{probability theory}
Suppose that the distribution of random variable $X$ is determined by its moments, that the $X_n$ have moments of all orders, and that
$\lim_n E[X_n^r]=E[X^r]$ for $r=1,2,\cdots$. Then $X_n\Rightarrow X$ (i.e., $X_n$ converges in distribution to $X$).
\end{lemma}
\section{An approximation of $S(t,f)$}
In this section we will prove several lemmas and derive an approximation of $S(t,f)$.
We denote by $\rho=\beta+i\gamma$ a typical zero of $L(s,f)$ inside the critical strip, i.e., $0<\beta<1$.
For $\Re(s)>1$, we have the Dirichlet series expansion for the logarithmic derivative of $L(s,f)$,
\bea\label{Log derivative definition}
-\frac{L'}{L}(s,f)=\sum_{n=1}^\infty \frac{\Lambda(n)C_f(n)}{n^s},
\eea
where $\Lambda(n)$ denotes the von Mangoldt function, and
\bea\label{Log derivative coefficient}
C_f(n)=\begin{cases}
\alpha_f(p)^m+\beta_f(p)^m, \,&\textit{ if }\, n=p^m\,\, \textit{for a prime}\,\,p,\\
0, \,& \textit{ otherwise}.
\end{cases}
\eea
\begin{lemma}\label{lemma 1}
Let $x>1$. For $s\neq \rho$, and $s\neq -m-\frac{1}{2}$, we have the following identity
\begin{equation*}
\begin{split}
\frac{L'}{L}(s,f)=&-\sum_{n\leq x^3} \frac{C_f(n)\Lambda_x(n)}{n^s}
-\frac{1}{\log^2x}\sum_\rho\frac{x^{\rho-s}(1-x^{\rho-s})^2}{(\rho-s)^3}\\
&-\frac{1}{\log^2x}\sum_{m=0}^\infty\frac{x^{-m-\frac{1}{2}-s}(1-x^{-m-\frac{1}{2}-s})^2}{(-m-\frac{1}{2}-s)^3},
\end{split}
\end{equation*}
where
\bna
\Lambda_x(n)=\begin{cases}
\Lambda(n), \,&\textit{if }\,\,n\leq x,\\
\Lambda(n)\frac{\log^2(x^3/n)-2\log^2(x^2/n)}{2\log^2x}, \,& \textit{if }\,\,x\leq n\leq x^2,\\
\Lambda(n)\frac{\log^2(x^3/n)}{2\log^2x}, \,& \textit{if }\,\,x^2\leq n\leq x^3,\\
0, \,& \textit{if }\,\,n\geq x^3.\\
\end{cases}
\ena
\end{lemma}
\begin{proof}
Recall the discontinuous integral
\bna
\frac{1}{2\pi i}\int_{(\alpha)}\frac{y^s}{s^3}\mathrm{d}s=
\begin{cases}
\frac{\log^2 y}{2}, \,&\textit{if }\,\, y\geq 1,\\
0, \,& \textit{if }\,\, 0<y\leq 1
\end{cases}
\ena
for $\alpha>0$.
It follows from \eqref{Log derivative definition} and \eqref{Log derivative coefficient} that
\bna
-\log^2x\sum_{n=1}^\infty \frac{C_f(n)\Lambda_x(n)}{n^s}=
\frac{1}{2\pi i}\int_{(\alpha)}\frac{x^u(1-x^u)^2}{u^3}\frac{L'}{L}(s+u,f)\mathrm{d}u,
\ena
where $\alpha=\max \{2, 1+\Re(s)\}$. By moving the line of integration all way to the left, we
pick up the residues at $u=0$, $u=\rho-s$ and $u=-m-\frac{1}{2}-s$ ($m= 0,1,2, \cdots$), and deduce that
\begin{equation*}
\begin{split}
&\frac{1}{2\pi i}\int_{(\alpha)}\frac{x^u(1-x^u)^2}{u^3}\frac{L'}{L}(s+u,f)\mathrm{d}u\\
=&\frac{L'}{L}(s,f)\log^2x+\sum_\rho\frac{x^{\rho-s}(1-x^{\rho-s})^2}{(\rho-s)^3}
+\sum_{m=0}^\infty\frac{x^{-m-\frac{1}{2}-s}(1-x^{-m-\frac{1}{2}-s})^2}{(-m-\frac{1}{2}-s)^3}.
\end{split}
\end{equation*}
Thus we complete the proof of lemma.
\end{proof}
\begin{lemma}\label{lemma 2}
For $s=\sigma+it$, $s'=\sigma'+it'$ such that $1/2\leq \sigma, \sigma'\leq 10$, $s\neq \rho$, $s'\neq \rho$,
we have
\bna\label{Im}
\Im\left(\frac{L'}{L}(s,f)-\frac{L'}{L}(s',f)\right)=\Im\sum_\rho\left(\frac{1}{s-\rho}-\frac{1}{s'-\rho}\right)+O(1),
\ena
and
\bea\label{Re}
\Re\frac{L'}{L}(s,f)=\sum_\rho\frac{\sigma-\beta}{(\sigma-\beta)^2+(t-\gamma)^2}+O(\log(|t|+q)).
\eea
\end{lemma}
\begin{proof}
By the Hadamard factorization theorem of the entire function $\Lambda(s,f)$, we have
\bna
\frac{L'}{L}(s,f)=b(f)+\sum_\rho\left(\frac{1}{s-\rho}+\frac{1}{\rho}\right)
-\frac{\Gamma'(s+1/2)}{\Gamma(s+1/2)}-\frac{1}{2}\log q+\log(2\pi),
\ena
for some $b(f)\in \mathbb{C}$ with $\Re(b(f))=-\Re\sum_\rho\frac{1}{\rho}$ (see \cite[Proposition 5.7 (3)]{IK}).
Note that for $z\not\in -\mathbb{N}$,
\bea\label{log derivative gamma}
\frac{\Gamma'}{\Gamma}(z)=-\gamma+\sum_{k\geq 1}\left(\frac{1}{k}-\frac{1}{k-1+z}\right),
\eea
where $\gamma$ is the Euler constant.
Using \eqref{log derivative gamma}, we get that for $\frac{1}{4}\leq u$,
\bna
\Im\frac{\Gamma'}{\Gamma}(u+iv)=\sum_{k\geq 1}\frac{v}{(k+u-1)^2+v^2}\ll 1.
\ena
And for $|\arg z|\leq \pi-\varepsilon$,
\bna\label{derivative gamma}
\frac{\Gamma'}{\Gamma}(z)=\log z-\frac{1}{2z}+O_\varepsilon\left(\frac{1}{|z|^2}\right),
\ena
thus we complete the proof of lemma.
\end{proof}
Let $x\geq 4$, we define
\bea\label{sigma_x}
\sigma_x=\sigma_{x,f}=\sigma_{x,f,t}:=\frac{1}{2}+2\max_\rho\left\{\left|\beta-\frac{1}{2}\right|, \frac{5}{\log x}\right\},
\eea
where $\rho=\beta+i\gamma$ runs through the zeros of $L(s,f)$ for which
\bea\label{condition 1}
|t-\gamma|\leq\frac{x^{3|\beta-1/2|}}{\log x}.
\eea
\begin{lemma}\label{lemma 3}
Let $x\geq 4$. For $\sigma\geq \sigma_x$, we have
\bna
\frac{L'}{L}(\sigma+it,f)=-\sum_{n\leq x^3} \frac{C_f(n)\Lambda_x(n)}{n^{\sigma+it}}
+O\left(x^{1/4-\sigma/2}\left|\sum_{n\leq x^3} \frac{C_f(n)\Lambda_x(n)}{n^{\sigma_x+it}}\right|\right)
+O\left(x^{1/4-\sigma/2}\log(|t|+q)\right),
\ena
and
\bna
\sum_\rho\frac{\sigma_x-1/2}{(\sigma_x-\beta)^2+(t-\gamma)^2}=
O\left(\left|\sum_{n\leq x^3} \frac{C_f(n)\Lambda_x(n)}{n^{\sigma_x+it}}\right|\right)
+O\big(\log(|t|+q)\big).
\ena
\end{lemma}
\begin{proof}
This proof follows \cite[Lemma 4.3]{LS} closely, and we give the details here for completeness.
By \eqref{Re}, we have
\bea\label{Re sigma_x}
\Re\frac{L'}{L}(\sigma_x+it,f)=\sum_\rho\frac{\sigma_x-\beta}{(\sigma_x-\beta)^2+(t-\gamma)^2}+O(\log(|t|+q)).
\eea
Moreover, if $\rho=\beta+i\gamma$ is zero of $L(s,f)$, then $(1-\beta)+i\gamma$ is also a zero of $L(s,f)$.
Thus we have
\begin{equation}\label{equation 1}
\begin{split}
&\sum_\rho\frac{\sigma_x-\beta}{(\sigma_x-\beta)^2+(t-\gamma)^2}\\
=&\frac{1}{2}\left(\sum_\rho\frac{\sigma_x-\beta}{(\sigma_x-\beta)^2+(t-\gamma)^2}
+\sum_\rho\frac{\sigma_x-(1-\beta)}{(\sigma_x-(1-\beta))^2+(t-\gamma)^2}\right)\\
=&\left(\sigma_x-\frac{1}{2}\right)\sum_\rho\frac{(\sigma_x-1/2)^2-(\beta-1/2)^2+(t-\gamma)^2}
{((\sigma_x-\beta)^2+(t-\gamma)^2)((\sigma_x-1+\beta)^2+(t-\gamma)^2)}.
\end{split}
\end{equation}
If $|\beta-1/2|\leq \frac{\sigma_x-1/2}{2}$, then
\begin{equation*}
\begin{split}
(\sigma_x-1/2)^2-(\beta-1/2)^2&\geq \frac{1}{2}\big((\sigma_x-1/2)^2+(\beta-1/2)^2\big)\\
&=\frac{1}{4}\big((\sigma_x-\beta)^2+(\sigma_x-1+\beta)^2\big).
\end{split}
\end{equation*}
Thus,
\bea\label{bound 1}
(\sigma_x-1/2)^2-(\beta-1/2)^2+(t-\gamma)^2\geq\frac{1}{4}\big((\sigma_x-1+\beta)^2+(t-\gamma)^2\big).
\eea
If $|\beta-1/2|> \frac{\sigma_x-1/2}{2}$, then by \eqref{sigma_x} and \eqref{condition 1} we have
\bna
|t-\gamma|>\frac{x^{3|\beta-1/2|}}{\log x}>3|\beta-1/2|.
\ena
Thus,
\begin{equation}\label{bound 2}
\begin{split}
&(\sigma_x-1/2)^2-(\beta-1/2)^2+(t-\gamma)^2\\
&=\big((\sigma_x-1/2)^2+(\beta-1/2)^2\big)+(t-\gamma)^2-2(\beta-1/2)^2\\
&\geq\frac{1}{2}\big((\sigma_x-\beta)^2+(\sigma_x-1+\beta)^2\big)+\frac{7}{9}(t-\gamma)^2\\
&\geq\frac{1}{4}\big((\sigma_x-1+\beta)^2+(t-\gamma)^2\big).
\end{split}
\end{equation}
Combining \eqref{bound 1}, \eqref{bound 2} and \eqref{equation 1}, we have
\bna
\sum_\rho\frac{\sigma_x-\beta}{(\sigma_x-\beta)^2+(t-\gamma)^2}
\geq\frac{1}{4}(\sigma_x-1/2)\sum_\rho\frac{1}
{((\sigma_x-\beta)^2+(t-\gamma)^2)}.
\ena
Using this bound and \eqref{Re sigma_x}, we obtain
\bea\label{step 1}
\sum_\rho\frac{1}{((\sigma_x-\beta)^2+(t-\gamma)^2)}
\leq \frac{4}{\sigma_x-1/2}\left|\frac{L'}{L}(\sigma_x+it,f)\right|
+O\left(\frac{\log(|t|+q)}{\sigma_x-1/2}\right).
\eea
On the other hand, by Lemma \ref{lemma 1}, we have
\begin{equation}\label{step 2}
\begin{split}
\frac{L'}{L}(\sigma+it,f)=&-\sum_{n\leq x^3} \frac{C_f(n)\Lambda_x(n)}{n^{\sigma+it}}
+\frac{w(x,\sigma,t)}{\log^2x}\sum_\rho\frac{x^{\beta-\sigma}(1+x^{\beta-\sigma})^2}
{\big((\sigma-\beta)^2+(t-\gamma)^2\big)^{3/2}}\\
&+O\left(\frac{x^{-1/2-\sigma}}{\log^2x}\right),
\end{split}
\end{equation}
with $|w(x,\sigma,t)|\leq1$.
Next we claim that
\bea\label{claim}
\frac{x^{\beta-\sigma}(1+x^{\beta-\sigma})^2}
{\big((\sigma-\beta)^2+(t-\gamma)^2\big)^{3/2}}
\leq 2\log x\frac{x^{1/4-\sigma/2}}{(\beta-\sigma_x)^2+(t-\gamma)^2}.
\eea
If $\beta\leq \frac{\sigma_x+1/2}{2}$, then we have
\begin{equation*}
\begin{split}
\frac{x^{\beta-\sigma}(1+x^{\beta-\sigma})^2}
{\big((\sigma-\beta)^2+(t-\gamma)^2\big)^{3/2}}&\leq
\frac{4x^{1/4-\sigma/2}}
{(\sigma_x-\beta)\big((\sigma_x-\beta)^2+(t-\gamma)^2\big)}\\
&\leq \frac{8}{\sigma_x-1/2}\frac{x^{1/4-\sigma/2}}
{\big((\sigma_x-\beta)^2+(t-\gamma)^2\big)}\\
&\leq \frac{4}{5}\log x\frac{x^{1/4-\sigma/2}}
{\big((\sigma_x-\beta)^2+(t-\gamma)^2\big)}.
\end{split}
\end{equation*}
If $\beta> \frac{\sigma_x+1/2}{2}$, then we have
\bna
|t-\gamma|>\frac{x^{3|\beta-1/2|}}{\log x}>3|\beta-1/2|\geq 3|\beta-\sigma_x|.
\ena
Thus, $(t-\gamma)^2>\frac{8}{9}\big((\beta-\sigma_x)^2+(t-\gamma)^2\big)$.
Hence
\begin{equation*}
\begin{split}
\frac{x^{\beta-\sigma}(1+x^{\beta-\sigma})^2}
{\big((\sigma-\beta)^2+(t-\gamma)^2\big)^{3/2}}&\leq
\frac{x^{\beta-\sigma}(1+x^{\beta-1/2})^2}
{|t-\gamma|(t-\gamma)^2}
\leq\frac{\log x}{x^{3|\beta-1/2|}}\frac{9}{8}
\frac{x^{\beta-\sigma}(1+x^{\beta-1/2})^2}
{(\beta-\sigma_x)^2+(t-\gamma)^2}\\
&= \frac{9}{8}(\log x)(1+x^{-(\beta-1/2)})^2\frac{x^{1/2-\sigma}}
{(\sigma_x-\beta)^2+(t-\gamma)^2}\\
&\leq \frac{9}{8}(\log x)(1+e^{-5})^2\frac{x^{1/2-\sigma}}
{(\sigma_x-\beta)^2+(t-\gamma)^2}\\
&<2(\log x)\frac{x^{1/2-\sigma}}
{(\sigma_x-\beta)^2+(t-\gamma)^2}.
\end{split}
\end{equation*}
Using \eqref{step 1} and \eqref{claim}, we have
\begin{equation*}
\begin{split}
&\sum_\rho\frac{x^{\beta-\sigma}(1+x^{\beta-\sigma})^2}
{\big((\sigma-\beta)^2+(t-\gamma)^2\big)^{3/2}}\\
\leq&\,8\log x\frac{x^{1/4-\sigma/2}}{\sigma_x-1/2}\left|\frac{L'}{L}(\sigma_x+it,f)\right|
+O\left(\frac{(\log x)x^{1/4-\sigma/2}\log(|t|+q)}{\sigma_x-1/2}\right)\\
\leq&\,\frac{4}{5}(\log x)^2x^{1/4-\sigma/2}\left|\frac{L'}{L}(\sigma_x+it,f)\right|
+O\left((\log x)^2x^{1/4-\sigma/2}\log(|t|+q)\right).
\end{split}
\end{equation*}
Inserting this bound into \eqref{step 2}, we have
\begin{equation}\label{step 3}
\begin{split}
\frac{L'}{L}(\sigma+it,f)=&-\sum_{n\leq x^3} \frac{C_f(n)\Lambda_x(n)}{n^{\sigma+it}}
+\frac{4}{5}w'(x,\sigma,t)x^{1/4-\sigma/2}\frac{L'}{L}(\sigma_x+it,f)\\
&+O\left(x^{1/4-\sigma/2}\log(|t|+q)\right),
\end{split}
\end{equation}
with $|w'(x,\sigma,t)|\leq 1$.
By taking $\sigma=\sigma_x$,
\begin{equation*}
\begin{split}
\left(1-\frac{4}{5}w'(x,\sigma_x,t)x^{1/4-\sigma_x/2}\right)
\frac{L'}{L}(\sigma_x+it,f)
=O\left(\left|\sum_{n\leq x^3} \frac{C_f(n)\Lambda_x(n)}{n^{\sigma_x+it}}\right|\right)
+O\left(x^{1/4-\sigma_x/2}\log(|t|+q)\right).
\end{split}
\end{equation*}
Since $\left|1-\frac{4}{5}w'(x,\sigma_x,t)x^{1/4-\sigma_x/2}\right|
\geq 1-\left|\frac{4}{5}w'(x,\sigma_x,t)x^{1/4-\sigma_x/2}\right|\geq \frac{1}{5}$,
we have
\bea\label{step 4}
\frac{L'}{L}(\sigma_x+it,f)
=O\left(\left|\sum_{n\leq x^3} \frac{C_f(n)\Lambda_x(n)}{n^{\sigma_x+it}}\right|\right)
+O\left(x^{1/4-\sigma_x/2}\log(|t|+q)\right).
\eea
Inserting \eqref{step 4} into \eqref{step 3} and \eqref{step 1}, respectively, we complete the proof of lemma.
\end{proof}
The following theorem provides an approximation of $S(t,f)$.
\begin{theorem}\label{theorem 1}
For $t\neq 0$, and $x\geq 4$, we have
\begin{equation*}
\begin{split}
S(t,f)=&\frac{1}{\pi}\Im\sum_{n\leq x^3}\frac{C_f(n)\Lambda_x(n)}{n^{\sigma_x+it}\log n}+
O\left((\sigma_x-1/2)\left|\sum_{n\leq x^3} \frac{C_f(n)\Lambda_x(n)}{n^{\sigma_x+it}}\right|\right)\\
&+O\big((\sigma_x-1/2)\log(|t|+q)\big),
\end{split}
\end{equation*}
where $\sigma_x$ is defined by \eqref{sigma_x}.
\end{theorem}
\begin{proof}
We have
\begin{equation*}
\begin{split}
\pi S(t,f)&=-\int_{1/2}^\infty\Im \frac{L'}{L}(\sigma+it,f)\mathrm{d}\sigma\\
&=-\int_{\sigma_x}^\infty\Im \frac{L'}{L}(\sigma+it,f)\mathrm{d}\sigma
-(\sigma_x-1/2)\Im \frac{L'}{L}(\sigma_x+it,f)\\
&\quad+\int_{1/2}^{\sigma_x}\Im\left( \frac{L'}{L}(\sigma_x+it,f)-\frac{L'}{L}(\sigma+it,f)\right)\mathrm{d}\sigma\\
&=:\, J_1+J_2+J_3.
\end{split}
\end{equation*}
By Lemma \ref{lemma 3}, we have
\begin{equation*}
\begin{split}
J_1=&-\int_{\sigma_x}^\infty\Im \frac{L'}{L}(\sigma+it,f)\mathrm{d}\sigma\\
=&\int_{\sigma_x}^\infty\Im\sum_{n\leq x^3} \frac{C_f(n)\Lambda_x(n)}{n^{\sigma+it}}\mathrm{d}\sigma
+O\left(\left|\sum_{n\leq x^3} \frac{C_f(n)\Lambda_x(n)}{n^{\sigma_x+it}}\right|
\int_{\sigma_x}^\infty x^{1/4-\sigma/2}\mathrm{d}\sigma\right)\\
&+O\left(\log(|t|+q)\int_{\sigma_x}^\infty x^{1/4-\sigma/2}\mathrm{d}\sigma\right)\\
=&\Im\sum_{n\leq x^3} \frac{C_f(n)\Lambda_x(n)}{n^{\sigma_x+it}\log n}
+O\left(\frac{1}{\log x}\left|\sum_{n\leq x^3} \frac{C_f(n)\Lambda_x(n)}{n^{\sigma_x+it}}\right|\right)
+O\left(\frac{\log(|t|+q)}{\log x}\right).
\end{split}
\end{equation*}
Taking $\sigma=\sigma_x$ in Lemma \ref{lemma 3}, we have
\begin{equation*}
\begin{split}
J_2\leq&\,(\sigma_x-1/2)\left|\frac{L'}{L}(\sigma_x+it,f)\right|\\
\ll&\,(\sigma_x-1/2)\left|\sum_{n\leq x^3} \frac{C_f(n)\Lambda_x(n)}{n^{\sigma_x+it}}\right|
+(\sigma_x-1/2)\log(|t|+q).
\end{split}
\end{equation*}
We apply Lemma \ref{lemma 2} to bound $J_3$ and get
\begin{equation*}
\begin{split}
&\Im\left( \frac{L'}{L}(\sigma_x+it,f)-\frac{L'}{L}(\sigma+it,f)\right)\\
=&\sum_\rho \frac{(\sigma_x-\sigma)(\sigma+\sigma_x-2\beta)(t-\gamma)}
{\big((\sigma_x-\beta)^2+(t-\gamma)^2\big)\big((\sigma-\beta)^2+(t-\gamma)^2\big)}
+O(1).
\end{split}
\end{equation*}
Hence
\begin{equation*}
\begin{split}
|J_3|\leq&\sum_\rho\int_{1/2}^{\sigma_x}
 \frac{(\sigma_x-\sigma)|\sigma+\sigma_x-2\beta||t-\gamma|}
{\big((\sigma_x-\beta)^2+(t-\gamma)^2\big)\big((\sigma-\beta)^2+(t-\gamma)^2\big)}\mathrm{d}\sigma
+O(\sigma_x-1/2)\\
\leq&\sum_\rho\frac{\sigma_x-1/2}{(\sigma_x-\beta)^2+(t-\gamma)^2}\int_{1/2}^{\sigma_x}
 \frac{|\sigma+\sigma_x-2\beta||t-\gamma|}
{(\sigma-\beta)^2+(t-\gamma)^2}\mathrm{d}\sigma
+O(\sigma_x-1/2).
\end{split}
\end{equation*}
If $|\beta-1/2|\leq\frac{1}{2}(\sigma_x-1/2)$, then for $1/2\leq\sigma\leq\sigma_x$,
\begin{equation*}
\begin{split}
|\sigma+\sigma_x-2\beta|&=|(\sigma-1/2)+(\sigma_x-1/2)-2(\beta-1/2)|\\
&\leq |\sigma-1/2|+|\sigma_x-1/2|+2|\beta-1/2|\leq 3(\sigma_x-1/2).
\end{split}
\end{equation*}
Thus,
\begin{equation*}
\begin{split}
\int_{1/2}^{\sigma_x}
 \frac{|\sigma+\sigma_x-2\beta||t-\gamma|}
{(\sigma-\beta)^2+(t-\gamma)^2}\mathrm{d}\sigma
&\leq 3(\sigma_x-1/2)\int_{-\infty}^\infty\frac{|t-\gamma|}
{(\sigma-\beta)^2+(t-\gamma)^2}\mathrm{d}\sigma\\
&\leq 10(\sigma_x-1/2).
\end{split}
\end{equation*}
If $|\beta-1/2|>\frac{1}{2}(\sigma_x-1/2)$, then
\bna
|t-\gamma|>\frac{x^{3|\beta-1/2|}}{\log x}>3|\beta-1/2|
\ena
and for $1/2\leq\sigma\leq\sigma_x$,
\begin{equation*}
\begin{split}
|\sigma+\sigma_x-2\beta|\leq |\sigma-1/2|+|\sigma_x-1/2|+2|\beta-1/2|\leq 6|\beta-1/2|.
\end{split}
\end{equation*}
Thus,
\begin{equation*}
\begin{split}
\int_{1/2}^{\sigma_x}
 \frac{|\sigma+\sigma_x-2\beta||t-\gamma|}
{(\sigma-\beta)^2+(t-\gamma)^2}\mathrm{d}\sigma
&<\int_{1/2}^{\sigma_x}\frac{|\sigma+\sigma_x-2\beta|}
{|t-\gamma|}\mathrm{d}\sigma\\
&\leq \int_{1/2}^{\sigma_x}\frac{6|\beta-1/2|}
{3|\beta-1/2|}\mathrm{d}\sigma=2(\sigma_x-1/2).
\end{split}
\end{equation*}
By Lemma \ref{lemma 3}, we obtain
\begin{equation*}
\begin{split}
|J_3|\leq&\,10(\sigma_x-1/2)\sum_\rho\frac{\sigma_x-1/2}{(\sigma_x-\beta)^2+(t-\gamma)^2}
+O(\sigma_x-1/2)\\
=&\,O\left((\sigma_x-1/2)\left|\sum_{n\leq x^3} \frac{C_f(n)\Lambda_x(n)}{n^{\sigma_x+it}}\right|\right)
+O\big((\sigma_x-1/2)\log(|t|+q)\big).
\end{split}
\end{equation*}
Thus we complete the proof of lemma.
\end{proof}
\begin{theorem}
Under the \textrm{GRH}, we have
\bna
S(t,f)\ll\frac{\log(|t|+q)}{\log\log(|t|+q)}.
\ena
\end{theorem}
\begin{proof}
We have $\sigma_x=\frac{1}{2}+\frac{10}{\log x}$ by mean of the GRH.
By \eqref{Deligne bound},
\bna
|C_f(p^m)\Lambda_x(p^m)|\ll \log p.
\ena
Thus,
\bna
\left|\Im\sum_{n\leq x^3}\frac{C_f(n)\Lambda_x(n)}{n^{\sigma_x+it}\log n}\right|
\ll \sum_{p\leq x^3}p^{-1/2}\ll \frac{x^{3/2}}{\log x},
\ena
and
\bna
(\sigma_x-1/2)\left|\sum_{n\leq x^3} \frac{C_f(n)\Lambda_x(n)}{n^{\sigma_x+it}}\right|
\ll \frac{1}{\log x}\sum_{p\leq x^3}p^{-1/2}\log p\ll \frac{x^{3/2}}{\log x}.
\ena
Then the theorem follows by taking $x=\big(\log(|t|+q)\big)^{2/3}$ in Theorem \ref{theorem 1}.
\end{proof}
\section{Proof of Theorems \ref{main-theorem} and \ref{main-corollary}}
\begin{lemma}\label{lemma 6}
Let $t>0$ be given. For prime number $q$ sufficiently large and $x=q^{\delta/3}$ with
$0<\delta<\frac{6c}{8n+3A}$, we have
\bna
\sum_{f\in H_2(q)}\frac{(\sigma_{x,f}-1/2)^{4n}x^{4n(\sigma_{x,f}-1/2)}}{4\pi \left<f,f\right>}
\ll_{t,n,\delta,A}\frac{1}{(\log q)^{4n}},
\ena
where $\sigma_{x,f}$ is defined by \eqref{sigma_x}, $c$ and $A$ are as in Lemma \ref{zero-density}.
\end{lemma}
\begin{proof}
By the definition of $\sigma_{x,f}$,
\begin{equation*}
\begin{split}
&(\sigma_{x,f}-1/2)^{4n}x^{4n(\sigma_{x,f}-1/2)}\\
\leq&\left(\frac{10}{\log x}\right)^{4n}x^{40n/\log x}
+2^{4n+1}\sum_{\beta>\frac{1}{2}+\frac{5}{\log x}\atop |t-\gamma|\leq\frac{x^{3|\beta-1/2|}}{\log x}}
(\beta-1/2)^{4n}x^{8n(\beta-1/2)}.
\end{split}
\end{equation*}
On the other hand,
\begin{equation*}
\begin{split}
&\sum_{\beta>\frac{1}{2}+\frac{5}{\log x}\atop |t-\gamma|\leq\frac{x^{3|\beta-1/2|}}{\log x}}
(\beta-1/2)^{4n}x^{8n(\beta-1/2)}\\
\leq&\,\sum_{j=5}^{\frac{1}{2}\lfloor\log x\rfloor}\left(\frac{j+1}{\log x}\right)^{4n}
x^{8n\frac{j+1}{\log x}}\sum_{\frac{1}{2}+\frac{j}{\log x}<\beta\leq\frac{1}{2}+\frac{j+1}{\log x}
\atop |t-\gamma|\leq\frac{x^{3|\beta-1/2|}}{\log x}}1\\
\leq&\,\frac{1}{(\log x)^{4n}}\sum_{j=5}^{\frac{1}{2}\lfloor\log x\rfloor}\left(j+1\right)^{4n}
e^{8n(j+1)}N\left(f;\frac{1}{2}+\frac{j}{\log x}, |t|+\frac{e^{3(j+1)}}{\log x}\right).
\end{split}
\end{equation*}
By Lemma \ref{zero-density}, there exists an absolute constant $A>1$ such that
\begin{equation*}
\begin{split}
&\sum_{f\in H_2(q)}\frac{1}{4\pi \left<f,f\right>}
\sum_{\beta>\frac{1}{2}+\frac{5}{\log x}\atop |t-\gamma|\leq\frac{x^{3|\beta-1/2|}}{\log x}}
(\beta-1/2)^{4n}x^{8n(\beta-1/2)}\\
\leq&\,\frac{1}{(\log x)^{4n}}\sum_{j=5}^{\frac{1}{2}\lfloor\log x\rfloor}\left(j+1\right)^{4n}
e^{8n(j+1)}\sum_{f\in H_2(q)}\frac{1}{4\pi \left<f,f\right>}
N\left(f;\frac{1}{2}+\frac{j}{\log x}, |t|+\frac{e^{3(j+1)}}{\log x}\right)\\
\ll&\,\frac{1}{(\log x)^{4n}}\sum_{j=5}^{\infty}\left(j+1\right)^{4n}
e^{8n(j+1)}\left(1+|t|+\frac{e^{3(j+1)}}{\log x}\right)^A q^{-2c\frac{j}{\log x}}\log q\\
\ll&_{t,n,A,\delta}\,\frac{\log q}{(\log x)^{4n+A}}\sum_{j=5}^{\infty}\left(j+1\right)^{4n}
e^{(8n+3A-\frac{6c}{\delta})j}\\
\ll&_{t,n,A,\delta}\,\frac{1}{(\log q)^{4n}}
\end{split}
\end{equation*}
provided that $0<\delta<\frac{6c}{8n+3A}$. 
In addition, by Lemma \ref{corollary},
\begin{equation*}
\begin{split}
\sum_{f\in H_2(q)}\frac{1}{4\pi \left<f,f\right>}
\left(\frac{10}{\log x}\right)^{4n}x^{40n/\log x}
\ll\frac{1}{(\log x)^{4n}}\ll\frac{1}{(\log q)^{4n}}.
\end{split}
\end{equation*}
Thus we complete the proof of lemma.
\end{proof}

For a parameter $x = q^{\delta/3}$ with a sufficiently small $\delta > 0$ (so that $\log x \asymp \log q$), let
\bna
M(t,f):=\frac{1}{\pi}\Im\sum_{p\leq x^3}\frac{C_f(p)}{p^{1/2+it}},\quad
\text{and}\quad
R(t,f):=S(t,f)-M(t,f),
\ena
where $C_f(p)=\alpha_f(p)+\beta_f(p)=\lambda_f(p)$.
\begin{lemma}\label{lemma 5} We have
\begin{equation*}
\begin{split}
R(t,f)&=O\left(\left|\Im\sum_{p\leq x^3}
\frac{C_f(p)(\Lambda_x(p)-\Lambda(p))}{p^{1/2+it}\log p}\right|\right)
+O\left(\left|\Im\sum_{p\leq x^{3/2}}\frac{C_f(p^2)\Lambda_x(p^2)}{p^{1+2it}\log p}\right|\right)\\
&+O\left((\sigma_x-1/2)x^{\sigma_x-1/2}\int_{1/2}^\infty
x^{1/2-\sigma}\left|\sum_{p\leq x^3}
\frac{C_f(p)\Lambda_x(p)\log(xp)}{p^{\sigma+it}}\right|\mathrm{d}\sigma\right)\\
&+O\big((\sigma_x-1/2)\log(|t|+q)\big)+O(1).
\end{split}
\end{equation*}
\end{lemma}
\begin{proof}
By Theorem \ref{theorem 1},
\begin{equation*}
\begin{split}
&R(t,f)=S(t,f)-M(t,f)\\=&\,\frac{1}{\pi}\Im\sum_{p\leq x^3}
\frac{C_f(p)(\Lambda_x(p)p^{1/2-\sigma_x}-\Lambda(p))}{p^{1/2+it}\log p}
+\frac{1}{\pi}\Im\sum_{m=2}^\infty\sum_{p^m\leq x^3}\frac{C_f(p^m)\Lambda_x(p^m)}{p^{m(\sigma_x+it)}\log p^m}\\
&+O\left((\sigma_x-1/2)\left|\sum_{m=1}^\infty\sum_{p^m\leq x^3} \frac{C_f(p^m)
\Lambda_x(p^m)}{p^{m(\sigma_x+it)}}\right|\right)
+O\big((\sigma_x-1/2)\log(|t|+q)\big).
\end{split}
\end{equation*}
Using the bound \eqref{Deligne bound}, we deduce that
\begin{equation*}
\begin{split}
\sum_{m=3}^\infty\sum_{p^m\leq x^3}\frac{C_f(p^m)\Lambda_x(p^m)}{p^{m(\sigma_x+it)}}=O(1),\quad\text{and}\quad
\sum_{m=3}^\infty\sum_{p^m\leq x^3}\frac{C_f(p^m)\Lambda_x(p^m)}{p^{m(\sigma_x+it)}\log p}=O(1).
\end{split}
\end{equation*}
Note that
\begin{equation*}
\begin{split}
(\sigma_x-1/2)\left|\sum_{p\leq x^{3/2}} \frac{C_f(p^2)
\Lambda_x(p^2)}{p^{2(\sigma_x+it)}}\right| &\ll
(\sigma_x-1/2)\sum_{p\leq x^{3/2}}\frac{\log p}{p}\\
&\ll(\sigma_x-1/2)\log x\ll (\sigma_x-1/2)\log q.
\end{split}
\end{equation*}
Thus,
\begin{equation*}
\begin{split}
R(t,f)&=\frac{1}{\pi}\Im\sum_{p\leq x^3}
\frac{C_f(p)(\Lambda_x(p)-\Lambda(p))}{p^{1/2+it}\log p}
-\frac{1}{\pi}\Im\sum_{p\leq x^3}
\frac{C_f(p)\Lambda_x(p)(1-p^{1/2-\sigma_x})}{p^{1/2+it}\log p}\\
&+\frac{1}{\pi}\Im\sum_{p\leq x^{3/2}}\frac{C_f(p^2)\Lambda_x(p^2)}{p^{1+2it}\log p^2}
+\frac{1}{\pi}\Im\sum_{p\leq x^{3/2}}\frac{C_f(p^2)\Lambda_x(p^2)(p^{1-2\sigma_x}-1)}{p^{1+2it}\log p^2}\\
&+O\left((\sigma_x-1/2)\left|\sum_{p\leq x^3} \frac{C_f(p)
\Lambda_x(p)}{p^{\sigma_x+it}}\right|\right)
+O\big((\sigma_x-1/2)\log(|t|+q)\big)+O(1).
\end{split}
\end{equation*}
Note that for $1/2\leq a\leq \sigma_x$,
\begin{equation*}
\begin{split}
\left|\sum_{p\leq x^3} \frac{C_f(p)\Lambda_x(p)}{p^{a+it}}\right|
&=x^{a-1/2}\left|\int_a^\infty x^{1/2-\sigma}\sum_{p\leq x^3}
\frac{C_f(p)\Lambda_x(p)\log(xp)}{p^{\sigma+it}}\mathrm{d}\sigma\right|\\
&\leq x^{\sigma_x-1/2}\int_{1/2}^\infty x^{1/2-\sigma}\left|\sum_{p\leq x^3}
\frac{C_f(p)\Lambda_x(p)\log(xp)}{p^{\sigma+it}}\right|\mathrm{d}\sigma.
\end{split}
\end{equation*}
Thus,
\begin{equation*}
\begin{split}
&\left|\sum_{p\leq x^3}
\frac{C_f(p)\Lambda_x(p)(1-p^{1/2-\sigma_x})}{p^{1/2+it}\log p}\right|
=\left|\int_{1/2}^{\sigma_x} \sum_{p\leq x^3} \frac{C_f(p)\Lambda_x(p)}{p^{a+it}}\mathrm{d}a\right|\\
&\leq (\sigma_x-1/2)x^{\sigma_x-1/2}\int_{1/2}^\infty x^{1/2-\sigma}\left|\sum_{p\leq x^3}
\frac{C_f(p)\Lambda_x(p)\log(xp)}{p^{\sigma+it}}\right|\mathrm{d}\sigma,
\end{split}
\end{equation*}
and
\begin{equation*}
\begin{split}
(\sigma_x-1/2)\left|\sum_{p\leq x^3} \frac{C_f(p)
\Lambda_x(p)}{p^{\sigma_x+it}}\right|\leq (\sigma_x-1/2)x^{\sigma_x-1/2}\int_{1/2}^\infty
x^{1/2-\sigma}\left|\sum_{p\leq x^3}
\frac{C_f(p)\Lambda_x(p)\log(xp)}{p^{\sigma+it}}\right|\mathrm{d}\sigma.
\end{split}
\end{equation*}
Moreover,
\begin{equation*}
\begin{split}
\left|\sum_{p\leq x^{3/2}}\frac{C_f(p^2)\Lambda_x(p^2)(p^{1-2\sigma_x}-1)}{p^{1+2it}\log p}\right|
&\ll \sum_{p\leq x^{3/2}}\frac{1}{p}(1-p^{1-2\sigma_x})\ll \sum_{p\leq x^{3/2}}(\sigma_x-1/2)\frac{\log p}{p}\\
&\ll(\sigma_x-1/2)\log q.
\end{split}
\end{equation*}
Thus we complete the proof of lemma.
\end{proof}
\begin{proposition}\label{main propsition}
Let $t>0$ be given. For prime number $q$ sufficiently large and $x=q^{\delta/3}$ with sufficiently small $\delta>0$, we have
\begin{equation}\label{M(t,f) moment}
\begin{split}
\sum_{f\in H_2(q)}\frac{M(t,f)^n}{4\pi \left<f,f\right>}=C_n(\log\log q)^{n/2}+O_{t,n}\big((\log\log q)^{n/2-1}\big),
\end{split}
\end{equation}
and
\begin{equation}\label{R(t,f) moment}
\begin{split}
\sum_{f\in H_2(q)}\frac{|R(t,f)|^{2n}}{4\pi \left<f,f\right>}=O_{t,n}(1).
\end{split}
\end{equation}
Here $C_n$ is defined by
\bna
C_n=\begin{cases}
\frac{n!}{(n/2)!(2\pi)^{n}}, \,&\textit{if }\, n\,\, \text{is even},\\
0, \,& \textit{if }\, n \,\,\text{is odd}.
\end{cases}
\ena
\end{proposition}
\begin{proof}
Recall that
\bna
M(t,f)=\frac{1}{\pi}\Im\sum_{p\leq x^3}\frac{C_f(p)}{p^{1/2+it}}=
\frac{-i}{2\pi}\left(\sum_{p\leq x^3}\frac{\lambda_f(p)}{p^{1/2+it}}
-\sum_{p\leq x^3}\frac{\lambda_f(p)}{p^{1/2-it}}\right).
\ena
Set $x=q^{\delta/3}$ for a suitably small $\delta>0$ (to be specified later). Thus we have
\bna
M(t,f)^n=\frac{(-i)^n}{(2\pi)^n}\left(\sum_{p\leq q^{\delta}}\frac{\lambda_f(p)}{p^{1/2+it}}
-\sum_{p\leq q^{\delta}}\frac{\lambda_f(p)}{p^{1/2-it}}\right)^n.
\ena
A general term in the expansion of the above has the form
\bea\label{general term}
\frac{\lambda_f(p_1)^{m(p_1)+n(p_1)}}{p_1^{m(p_1)(1/2+it)}(-1)^{n(p_1)}p_1^{n(p_1)(1/2-it)}}
\times\cdots\times
\frac{\lambda_f(p_r)^{m(p_r)+n(p_r)}}{p_r^{m(p_r)(1/2+it)}(-1)^{n(p_r)}p_r^{n(p_r)(1/2-it)}},
\eea
where $p_1<p_2<\cdots<p_r\leq q^\delta$, $m(p_j)+n(p_j)\geq 1$, and $\sum_{j=1}^r(m(p_j)+n(p_j))=n$.

Now we discuss the contribution from the general term \eqref{general term} in the following cases.

Case\,(\uppercase\expandafter{\romannumeral 1}):
\emph{In the general term \eqref{general term}, $m(p_{j_0})\not\equiv n(p_{j_0})\bmod2$ for some $j_0$.}
With the Hecke relation \eqref{the Hecke relation}, the numerator of this general term \eqref{general term}
can be written as a sum whose terms are all of the form $\lambda_f(m)\lambda_f(n)$ where $m\neq n$.
Since for these terms $\delta_{m,n}=0$, we see from Lemma \ref{corollary} that the contribution from the general term
\eqref{general term} to the moment \eqref{M(t,f) moment} is
\begin{equation*}
\begin{split}
&\ll_n\sum_{p_1<p_2<\cdots<p_r\leq q^\delta}p_1^{-\frac{1}{2}(m(p_1)+n(p_1))}\cdots
p_r^{-\frac{1}{2}(m(p_r)+n(p_r))}q^{-\frac{3}{2}+\frac{\delta n}{2}(1+\varepsilon)}\\
&\ll_n\bigg(\sum_{p\leq q^\delta}p^{-\frac{1}{2}}\bigg)^rq^{-\frac{3}{2}+\frac{\delta n}{2}(1+\varepsilon)}
\ll_nq^{-\frac{3}{2}+\frac{\delta n}{2}(2+\varepsilon)},
\end{split}
\end{equation*}
which is negligible by taking $0<\delta<\frac{1}{n}$.

Case\,(\uppercase\expandafter{\romannumeral 2}):
\emph{In the general term \eqref{general term}, $m(p_{j})\equiv n(p_{j})\bmod2$ for all $j$,
and $m(p_{j_0})+n(p_{j_0})\geq4$ for some $j_0$.}
In this case, $n$ is even and $n>2r$, so $r<\frac{n}{2}$.
By using Hecke relations \eqref{the Hecke relation},
one sees that it is possible to have resulting terms $\lambda_f(m)\lambda_f(n)$ with $m=n$.
Using Lemma \ref{corollary}, the contribution to the moment \eqref{M(t,f) moment} such terms is bounded by
the diagonal term of size $\asymp 1$. The contribution from the off-diagonal terms is clearly negligible compared
to $1$, given that $0<\delta<\frac{1}{n}$.
Combining the bound \eqref{Deligne bound} and \eqref{classical bound 1},
we have the contribution from this general term to the moment \eqref{M(t,f) moment} is
\begin{equation*}
\begin{split}
&\ll_n\sum_{p_1<p_2<\cdots<p_r\leq q^\delta}p_1^{-\frac{1}{2}(m(p_1)+n(p_1))}\cdots
p_r^{-\frac{1}{2}(m(p_r)+n(p_r))}\\
&\ll_n\bigg(\sum_{p\leq q^\delta}p^{-1}\bigg)^r
\ll_n (\log\log q)^{\frac{n}{2}-1}.
\end{split}
\end{equation*}

It remains to discuss the following: In the general term \eqref{general term}, $m(p_{j})\equiv n(p_{j})\bmod2$
and $m(p_j)+n(p_j)=2$ for all $j$. Clearly $n$ must be even, say $n=2m$ and so $r=m$.

Case\,(\uppercase\expandafter{\romannumeral 3}):
\emph{$m(p_{j_0})=2$ or $n(p_{j_0})=2$ for some $j_0$.}
Under the assumption $0<\delta<\frac{1}{n}$, we deduce that the contribution from these terms to \eqref{M(t,f) moment} is
at most $O\big((\log\log q)^{\frac{n}{2}-1}\big)$ by using \eqref{Deligne bound}
and the convergence of $\sum_{p}p^{-1-it}$ for $t\neq 0$.

Case\,(\uppercase\expandafter{\romannumeral 4}):
\emph{$m(p_{j})=n(p_{j})=1$ for all $j$.}
The contribution from these terms to \eqref{M(t,f) moment} is
\begin{equation*}
\begin{split}
\sum_{f\in H_2(q)}\frac{1}{4\pi \left<f,f\right>}\binom{2m}{m}(m!)(m!)\frac{(-1)^m(-i)^{2m}}{(2\pi)^{2m}}
\sum_{p_1<p_2<\cdots<p_m\leq q^\delta}\frac{\lambda_f(p_1)^2\cdots\lambda_f(p_m)^2}{p_1\cdots p_m}.
\end{split}
\end{equation*}
By the Hecke relations, we have $\lambda_f(p_1)^2\cdots\lambda_f(p_m)^2=\lambda_f(p_1\cdots p_m)\lambda_f(p_1\cdots p_m)$.
Applying Lemma \ref{corollary} and the assumption $0<\delta<\frac{1}{n}$, we have the above equals
\begin{equation*}
\begin{split}
&\frac{(2m)!}{(2\pi)^{2m}}\sum_{p_1<p_2<\cdots<p_m\leq q^\delta}\frac{1}{p_1\cdots p_m}
\cdot\left(1+O\left(q^{-3/2}(p_1\cdots p_m)(\tau (p_1\cdots p_m))\right)\right)\\
=&\frac{(2m)!}{m!(2\pi)^{2m}}\sum_{p_1,p_2,\ldots,p_m\leq q^\delta\atop p_j \,\text{distinct}}\frac{1}{p_1\cdots p_m}
+O\bigg(q^{-3/2}\sum_{p_1,p_2,\ldots,p_m\leq q^\delta\atop p_j\,\text{distinct}}1\bigg)\\
=&\frac{(2m)!}{m!(2\pi)^{2m}}(\log\log q)^m
+O_m\big((\log\log q)^{m-1}\big).
\end{split}
\end{equation*}
Now the asymptotic formula \eqref{M(t,f) moment} follows from
Cases (\uppercase\expandafter{\romannumeral 1})--(\uppercase\expandafter{\romannumeral 4})
under the assumption $0<\delta<\frac{1}{n}$.

Next we are ready to prove \eqref{R(t,f) moment}.
Recall that $\sigma_x=\sigma_{x,f}$ depending on $f$.
By Lemma \ref{lemma 5},
\begin{equation}\label{R}
\begin{split}
&\sum_{f\in H_2(q)}\frac{|R(t,f)|^{2n}}{4\pi \left<f,f\right>}\\
\ll&\sum_{f\in H_2(q)}\frac{1}{4\pi \left<f,f\right>}\left|\Im\sum_{p\leq x^3}
\frac{C_f(p)(\Lambda_x(p)-\Lambda(p))}{p^{1/2+it}\log p}\right|^{2n}
+\sum_{f\in H_2(q)}\frac{1}{4\pi \left<f,f\right>}
\left|\Im\sum_{p\leq x^{3/2}}\frac{C_f(p^2)\Lambda_x(p^2)}{p^{1+2it}\log p}\right|^{2n}\\
&+\sum_{f\in H_2(q)}\frac{(\sigma_x-1/2)^{2n}x^{2n(\sigma_x-1/2)}}{4\pi \left<f,f\right>}\left(\int_{1/2}^\infty
x^{1/2-\sigma}\left|\sum_{p\leq x^3}
\frac{C_f(p)\Lambda_x(p)\log(xp)}{p^{\sigma+it}}\right|\mathrm{d}\sigma\right)^{2n}\\
&+\sum_{f\in H_2(q)}\frac{(\sigma_x-1/2)^{2n}}{4\pi \left<f,f\right>}\big(\log(|t|+q)\big)^{2n}
+\sum_{f\in H_2(q)}\frac{1}{4\pi \left<f,f\right>}.
\end{split}
\end{equation}
Since
\bna
|\Lambda_x(p)-\Lambda(p)|=O\left(\frac{\log^3p}{\log^2x}\right)\quad\text{and}\quad
C_f(p^2)=\lambda_f(p^2)-\varepsilon_q(p),
\ena
the first two terms are of $O(1)$ by the argument as \eqref{M(t,f) moment}.
The last two terms are of $O(1)$ by Lemma \ref{lemma 6} and Lemma \ref{corollary}.
For the third term,
it follows from Cauchy's inequality that
\begin{equation*}
\begin{split}
&\sum_{f\in H_2(q)}\frac{(\sigma_x-1/2)^{2n}x^{2n(\sigma_x-1/2)}}{4\pi \left<f,f\right>}\left(\int_{1/2}^\infty
x^{1/2-\sigma}\left|\sum_{p\leq x^3}
\frac{C_f(p)\Lambda_x(p)\log(xp)}{p^{\sigma+it}}\right|\mathrm{d}\sigma\right)^{2n}\\
&\leq \left(\sum_{f\in H_2(q)}\frac{(\sigma_x-1/2)^{4n}x^{4n(\sigma_x-1/2)}}{4\pi \left<f,f\right>}\right)^{1/2}\\
&\times\left(\sum_{f\in H_2(q)}\frac{1}{4\pi \left<f,f\right>}\left(\int_{1/2}^\infty
x^{1/2-\sigma}\left|\sum_{p\leq x^3}
\frac{C_f(p)\Lambda_x(p)\log(xp)}{p^{\sigma+it}}\right|\mathrm{d}\sigma\right)^{4n}\right)^{1/2}.
\end{split}
\end{equation*}
By H\"{o}lder's inequality with the exponents $4n/(4n-1)$ and $4n$,
\begin{equation*}
\begin{split}
&\left(\int_{1/2}^\infty
x^{1/2-\sigma}\left|\sum_{p\leq x^3}
\frac{C_f(p)\Lambda_x(p)\log(xp)}{p^{\sigma+it}}\right|\mathrm{d}\sigma\right)^{4n}\\
\leq&\left(\int_{1/2}^\infty
x^{1/2-\sigma}\mathrm{d}\sigma\right)^{4n-1}\left(\int_{1/2}^\infty
x^{1/2-\sigma}\left|\sum_{p\leq x^3}
\frac{C_f(p)\Lambda_x(p)\log(xp)}{p^{\sigma+it}}\right|^{4n}\mathrm{d}\sigma\right)\\
=&\frac{1}{(\log x)^{4n-1}}\int_{1/2}^\infty
x^{1/2-\sigma}\left|\sum_{p\leq x^3}
\frac{C_f(p)\Lambda_x(p)\log(xp)}{p^{\sigma+it}}\right|^{4n}\mathrm{d}\sigma.
\end{split}
\end{equation*}
Moreover, using the same argument as in proving \eqref{M(t,f) moment},
we have
\begin{equation*}
\begin{split}
&\sum_{f\in H_2(q)}\frac{1}{4\pi \left<f,f\right>}\left|\sum_{p\leq x^3}
\frac{C_f(p)\Lambda_x(p)\log(xp)}{p^{\sigma+it}}\right|^{4n}\\
\ll&\left(\sum_{p\leq x^3}
\frac{\left|\Lambda_x(p)\log(xp)p^{1/2-\sigma}\right|^2}{p}\right)^{2n}
\ll \left(\sum_{p\leq x^3}
\frac{(\log p)^2}{p}(\log x)^2\right)^{2n}
\ll(\log x)^{8n}.
\end{split}
\end{equation*}
Thus,
\begin{equation*}
\begin{split}
&\sum_{f\in H_2(q)}\frac{1}{4\pi \left<f,f\right>}\left(\int_{1/2}^\infty
x^{1/2-\sigma}\left|\sum_{p\leq x^3}
\frac{C_f(p)\Lambda_x(p)\log(xp)}{p^{\sigma+it}}\right|\mathrm{d}\sigma\right)^{4n}\\
\ll &\frac{1}{(\log x)^{4n-1}}\int_{1/2}^\infty
x^{1/2-\sigma}(\log x)^{8n}\mathrm{d}\sigma\ll (\log x)^{4n}.
\end{split}
\end{equation*}
Using this bound and Lemma \ref{lemma 6}, we have the third term in \eqref{R}
is of $O(1)$.
Thus we complete the proof of proposition.
\end{proof}

\noindent\textbf{Proof of Theorem \ref{main-theorem}.}
By the binomial theorem, we have
\begin{equation*}
\begin{split}
S(t,f)^n=M(t,f)^n+O_n\left(\sum_{\ell=1}^n|M(t,f)|^{n-\ell}|R(t,f)|^{\ell}\right).
\end{split}
\end{equation*}
For $1\leq\ell\leq n$, we apply the generalized H\"{o}lder's inequality and Proposition \ref{main propsition}, and get
\begin{equation*}
\begin{split}
&\sum_{f\in H_2(q)}\frac{1}{4\pi \left<f,f\right>}|M(t,f)|^{n-\ell}|R(t,f)|^{\ell}\\
\ll &\left(\sum_{f\in H_2(q)}\frac{1}{4\pi \left<f,f\right>}\right)^{\frac{1}{2}}
\left(\sum_{f\in H_2(q)}\frac{|M(t,f)|^{2n}}{4\pi \left<f,f\right>}\right)^{\frac{n-\ell}{2n}}
\left(\sum_{f\in H_2(q)}\frac{|R(t,f)|^{2n}}{4\pi \left<f,f\right>}\right)^{\frac{\ell}{2n}}\\
\ll&_{t,n}(\log\log q)^{(n-\ell)/2}\ll_{t,n}(\log\log q)^{(n-1)/2}.
\end{split}
\end{equation*}
\\
\noindent\textbf{Proof of Theorem \ref{main-corollary}.}
The $n$-th moment of $\mu_q$ is
\begin{align*}
\int_{\mathbb{R}}\xi^n\mathrm{d}\mu_q=\left(
\sum_{f\in H_2(q)}\omega_f\cdot\left(\frac{S(t,f)}{\sqrt{\log\log q}}\right)^n\right)\bigg/
\left(\sum_{f\in H_2(q)}\omega_f\right).
\end{align*}
Using Theorem \ref{main-theorem} and Lemma \ref{corollary}, we obtain that for all $n$,
\bna
\lim_{q\rightarrow \infty}\int_{\mathbb{R}}\xi^n\mathrm{d}\mu_q=C_n=
\sqrt{\pi}\int_{\mathbb{R}}\xi^n \exp(-\pi^2\xi^2)\mathrm{d}\xi.
\ena
By the theory of moments in probability theory (see Lemma \ref{probability theory}),
we complete the proof of theorem.

\bigskip
\noindent{\bf Acknowledgements}
H. Wang would like to thank the Alfr\'{e}d R\'{e}nyi Institute of Mathematics for providing a great working environment
and the China Scholarship Council for its financial support.
The authors are very grateful to the
referees for their valuable suggestions.

\bibliographystyle{amsplain}

\end{document}